\documentclass{article}

\usepackage[english]{babel}
\usepackage{csquotes}

\usepackage{amsmath}
\usepackage{graphicx}
\usepackage{amssymb}
\usepackage{amsthm}
\usepackage{mathrsfs}
\usepackage{graphicx}
\usepackage{authblk}
\usepackage{mathtools}
\usepackage{biblatex}

\usepackage{tikz}
\usetikzlibrary{trees}

\usepackage{bbm}

\usepackage{tikz-cd} 
\usepackage{csquotes}

\usepackage{hyperref}
\hypersetup{
    colorlinks=true,
    linkcolor=magenta,
    filecolor=blue,
    urlcolor=blue,
    citecolor=blue
    }
\usepackage{authblk}

\numberwithin{equation}{section}
\numberwithin{figure}{section}
\theoremstyle{plain}
\newtheorem{theorem}{Theorem}[section]
\newtheorem{proposition}[theorem]{Proposition}
\newtheorem{lemma}[theorem]{Lemma}
\newtheorem{corollary}[theorem]{Corollary}

\theoremstyle{definition}
\newtheorem{remark}[theorem]{Remark}
\newtheorem{example}[theorem]{Example}
\newtheorem{definition}[theorem]{Definition}

\numberwithin{equation}{section}

\usepackage{csquotes}

\usepackage{thmbox}
\newtheorem[S, bodystyle=\normalfont\noindent]{defiS}{Definition}[section]

\title{Introduction to Elliptic Quasi-Modular Forms via Moduli Spaces\footnote{MSC2020: 11F11, 14C30, 11G05, 14D22, 14J33}}
\author[1]{Walter Andrés Páez Gaviria}

\affil[1]{Instituto de Matemática Pura e Aplicada. Rio de Janeiro, Brasil.}

\affil[1]{Institución Universitaria Politécnico Grancolombiano. Bogotá, Colombia.}

\date{}

\begin{document}

\maketitle
\begin{abstract}
In this paper we present rigorously and as succintly as possible the theory of elliptic quasi-modular forms by means of moduli spaces and the Gauss-Manin connection, and deal with one of the main historical appearances of quasi-modular forms, which was the seminal case of mirror symmetry treated by Dijkgraaf.
\end{abstract}

\section{Introduction}

There are many great introductions to the theory of modular forms, for instance \parencite{123,diamond,koblitz1993,serre}. Quasi-modular forms were first considered by Kaneko and Zagier in \parencite{Kaneko}, and since then they have appeared in many different contexts, for instance in \parencite{mov2008,mov2012,pascal,dijkgraaf,ochiai}. Elementary expositions of the theory of quasi-modular forms can be found in \parencite[First lecture notes]{123} and \parencite{pascal}. In this paper we present rigorously and as succintly as possible the theory of elliptic quasi-modular forms by means of moduli spaces and the Gauss-Manin connection, and deal with one of the main historical appearances of quasi-modular forms, which was the seminal case of mirror symmetry treated by Dijkgraaf.

The approach used in this paper is part of a more general project called the Gauss Manin Connection in Disguise (GMCD), developed by Hossein Movasati and its collaborators, whose main goal is ``to unify  automorphic forms with topological string partition functions and give a systematic way of studying sub-moduli  of moduli spaces using the theory of holomorphic/algebraic  foliations. It aims to replace the variation of Hodge structures with certain vector fields whose solutions are a vast generalization of automorphic forms. It aims to return Hodge theory to its origin which is the study of multiple integrals due to Abel, Gauss, Legendre, Poincaré, Picard and many others."

For further cases of the GMCD project, the reader is referred to \parencite{Movasati2015Hecke,AlimMovasatiScheideggerYau2016,Movasati2017NoetherLefschetz,HaghighatMovasatiYau2017,MovasatiNikdelan2021Dwork,AlimVogrin2021GaussManinLieK3,CaoMovasatiVillaflorLoyola2024,Paez2025}.

\section{Quasi-modular forms and Ramanujan's identities}

Let $k$ be a non-negative integer.

\begin{definition}
A \textbf{modular form for $SL(2,\mathbb{Z})$ of weight $k$} is a holomorphic function $f:\mathbb{H}\rightarrow \mathbb{C}$ satisfying:
\begin{itemize}
\item (automorphy condition) $f(Mz)=(cz+d)^kf(z)$ for all $M=\begin{psmallmatrix}a & b\\c & d\end{psmallmatrix}\in SL(2,\mathbb{Z})$.

In particular, using $M=\begin{psmallmatrix}1 & 1\\0 & 1\end{psmallmatrix}\in SL(2,\mathbb{Z})$, we get $f(z+1)=f(z)$.
\item (holomorphic at $i\infty$) Since $f$ is periodic with period $1$, it has a Fourier expansion $f(z)=\sum_{n=-\infty}^{\infty} a_n q^n$, where $q=e^{2\pi iz}$. We require that $a_n=0$ for every $n<0$. 
\end{itemize}

\noindent   We denote by $\mathfrak{M}_k$ the set of all modular forms of weight $k$. It is straightforward to see that $\mathfrak{M}=\oplus_{k\geq 0}\mathfrak{M}_k$ has the structure of a graded $\mathbb{C}$-algebra.
\end{definition}

Observe that, by considering $M=\begin{psmallmatrix}-1 & 0\\0 & -1\end{psmallmatrix}\in SL(2,\mathbb{Z})$ in the automorphy condition, every modular form of odd weight is identically zero. 

The main examples of modular forms are the Eisenstein series  $E_k:\mathbb{H}\rightarrow \mathbb{C}$ defined as follows:

\begin{equation}\label{bernoulli}E_{k}(z)=1-\frac{2k}{B_k}\sum_{n=1}^{\infty} \sigma_{k-1}(n)q^n.
\end{equation}

Here $\sigma_i(n)=\sum_{d|n}d^i$.  For $k>2$, it can be proved that \parencite[First lecture notes, \S 2.1]{123}

\begin{equation}
E_k(z)=\frac{1}{2 \zeta(k)}
\sum_{\substack{(m,n)\in \mathbb{Z}\times \mathbb{Z} \\ (m,n)\neq (0,0)}}
\frac{1}{(mz+n)^k}.
\end{equation}

From the previous expression  we get that, for $k>2$, the Eisenstein series satisfy

\begin{equation} \label{eisensteinautomorphy}
E_k(\gamma\cdot z)=(cz+d)^kE_k(z)
\end{equation}for every $\gamma=\begin{psmallmatrix} a & b \\ c & d \end{psmallmatrix}\in \text{SL}_2(\mathbb{Z})$. From equations \eqref{bernoulli} and \eqref{eisensteinautomorphy},  we get that $E_k\in \mathfrak{M}_k$ for $k>2$.

As for the Eisenstein series, for many cases the coefficients of the $q$-expansion of a modular form have some kind of ``arithmetic" information. Furthermore, the structure of the algebra $\mathfrak{M}$ is well-known:

\begin{theorem} $\mathfrak{M} $ is equal to $\mathbb{C}[E_4,E_6]$. In other words, every modular form can be written as a polynomial in the Eisenstein series $E_4$ and $E_6$. Furthermore, $E_4$ and $E_6$ are algebraically independent.
\end{theorem}

\noindent What happens with $E_2$ is quite interesting. It can be shown that 
\begin{equation}
E_2(z)=1+\frac{6}{\pi^2}\sum_{m=1}^{\infty}\sum_{n=-\infty}^{\infty}\frac{1}{(mz+n)^2}.
\end{equation}

\noindent Therefore, $E_2(z+1)=E_2(z)$. On the other, the other automorphy condition does not hold. It can be proved that \parencite[First lecture notes, Proposition 6]{123}
\begin{equation}z^{-2}E_2(\frac{-1}{z})=E_2(z)+\frac{12}{2\pi i z}.\end{equation}
    
In \parencite[First lecture notes, \S5.1]{123} we find the following definition (whose origin dates back to \parencite{Kaneko})
    
\begin{definition}
The graded algebra $\tilde{\mathfrak{M}}$ of quasimodular forms for $SL(2,\mathbb{Z})$ is defined to be $\mathbb{C}[E_2,E_4,E_6]$.
\end{definition}

It must be admitted that this is an extremely \textit{ad hoc} definition. So, one may ask: Why are quasimodular forms interesting or useful?

A hint of the interest of quasi-modular forms is given by the following example:

\begin{example} Let $f\in\mathfrak{M}_k$. Let us define \begin{equation}\label{differentiation}f':=\frac{1}{2\pi i}\frac{df}{dz}=\sum_{n\geq 1}n a_n q^{n-1}.
\end{equation}
            
Then, $$f'(\frac{az+b}{cz+d})=(cz+d)^{k+2}f'(z)+\frac{k}{2\pi i}c(cz+d)^{k+1}f(z).$$
            
So, $f'$ is not modular! But using this computation, and the previous proposition about $E_2$, it is easy to see that $$f'-\frac{k}{12}E_2 f\in\mathfrak{M}_{k+2}.$$
            
Therefore, $f'\in\mathbb{C}[E_2,E_4,E_6]_{k+2}$.

\end{example}
By means of the previous example we see that the differentiation defined in e
quation \eqref{differentiation} defines a structure of differential graded $\mathbb{C}$-algebra on $\tilde{\mathfrak{M}}$. Furthermore, the structure of the differential algebra $\tilde{\mathfrak{M}}$ is determined by means of the following theorem, whose first proof is due to Ramanujan \parencite[Equation 30, page 181]{ramanujan}. 

\begin{theorem}\label{ramanujan}

    The algebra $\mathfrak{\tilde{M}}$ is closed under differentiation. More specifically, we have
    
    $$E_2'=\frac{E_2^2-E_4}{12},$$
    $$E_4'=\frac{E_2 E_4-E_6}{3},$$
    $$E_6'=\frac{E_2E_6-E_4^2}{2}.$$

\end{theorem}

In this short paper we will give an account of the previous results by means of moduli spaces and the Gauss-Manin connection for elliptic curves, following \parencite{mov2008,mov2012,pascal}. Furthermore, we will also explain in full detail the following example, which is the manifestation of mirror symmetry in the elliptic curve case, first studied by Dijkgraaf in \parencite{dijkgraaf}.
\begin{example}
Let $E$ be an elliptic curve over $\mathbb{C}$, and $g\geq 1, m\geq 2$ integers.

A pair $(C,f)$ consisting of a smooth complex curve $f$ and a holomorphic map $f:C\rightarrow E$ is called an \textit{simple $m$-branched cover} if $C$ is connected, every ramification point has index $m$, and different ramification points have different images.

Let $X_{g,d}^{(m)}$ be the set of isomorphism classes of $m$-branched covers of $E$ of degree $d$ and genus $g$, and let $$N_{g,d}^{(m)}=\sum_{(C,f)\in X_{g,d}^{(m)}} \frac{1}{|Aut{(C,f)}|}.$$

Define $$F_g^{(m)}(q)=\sum_{d\geq 1}N_{g,d}^{(m)}q^d.$$

\begin{theorem}\parencite{dijkgraaf,ochiai} For $g\geq 2$, $F_g^{(m)}(q)\in \mathfrak{\tilde{M}}$. Furthermore, $F_g^{(2)}(q)\in \mathfrak{\tilde{M}}_{6g-g}.$
\end{theorem}
\end{example}

\section{Enhanced elliptic curves and the Ramanujan vector field}

In \parencite{mov2008,mov2012,pascal}, Movasati was able to recover the algebra $\mathfrak{M}_{*}$ of quasi-modular forms and its differential structure, i.e., the Ramanujan's relations \ref{ramanujan} by means of considering a moduli space of enhanced elliptic curves $\mathsf{T}$, and an algebraic vector field $\mathsf{R}$ defined on it, called the Ramanujan or modular vector field, which contains the information of the Ramanujan's relations.

\begin{definition}\label{enhanced}
    Let $\mathsf{T}$ denote the moduli of triples $(E,\alpha,\beta)$ where:
    \begin{itemize}
    \item $E$ is a complex elliptic curve;
    \item $\alpha\in F^1H_{dR}^1(E/\mathbb{C})$ and $\beta\in H_{dR}^1(E/\mathbb{C})$ are such that $\langle \alpha,\beta\rangle=1$.
    \end{itemize}
    
    Such a triple is called an \textbf{enhanced elliptic curve}.
    \end{definition}

We begin by recalling the well-known Weierstrass form of a complex elliptic curve:
\begin{theorem}\label{weierstrass}
Every complex elliptic curve is isomorphic to a hypersurface  $E_{t_1,t_2}$ in $\mathbb{P}^2$ defined in the affine chart given by $z\neq0$ by the zero locus of the polynomial 

\begin{equation}
F_{t_2,t_3}=y^2-x^3+t_2x+t_3,\,\,\,\,   \Delta:=27t_3^2-t_2^3\neq 0.
\end{equation}

Furthermore, $E_{t_1,t_2}\cong E_{t_1',t_2'}$ if, and only if, there is some $\kappa\in\mathbb{C}^*$ such that $t_i'=\kappa^it_i$ for $i=2,3$.
\end{theorem}

\begin{lemma}\label{automorphism} 
Every automorphism of the elliptic curve $E_{t_1,t_2}$ is, in the affine chart given by $z\neq 0$, of the form $(x,y)\mapsto(\lambda^2x,\lambda^3y),$ where $\lambda\in \mathbb{C}^*$ is such that $\lambda^{2i}t_i=t_i$ for $i=1,2$.
\end{lemma}
\begin{proof}
This is a straightforward computation.
\end{proof}

\begin{lemma}\label{second}
Let $k\in \mathbb{C}^{*}$, and $\phi:\mathbb{P}^2\rightarrow \mathbb{P}^2$ be the projective isomorphism given by $\phi
([x,y,z])=[k^2x,k^3y,z].$ Then, $\phi$ sends $E_{k^{-4}t_2,k^{-6}t_3}$ onto $E_{t_2,t_3}$. Furthermore, $\phi^*(\frac{dx}{y})=k^{-1}\frac{dx}{y}$ and $\phi^*(\frac{xdx}{y})=k\frac{xdx}{y}$.
\end{lemma}
\begin{proof}
This is a straightforward computation.
\end{proof}

\begin{theorem}
$\mathsf{T}$ is isomorphic to  $\text{Spec}(\mathbb{C}[t_1,t_2,t_3,\frac{1}{27t_3^2-t_2^3}])$. 
\end{theorem}

\begin{proof}

Let us observe that for each $(t_2,t_3)\in \mathbb{C}$ with $\Delta\neq 0$, we have an enhanced elliptic curve $(E_{t_2,t_3},\frac{dx}{y},\frac{xdx}{y})$, since $\langle \frac{dx}{y},\frac{xdx}{y}\rangle=1$ (for the last equality see \parencite[\S 2.10]{mov12}). Then, every other enhanced elliptic curve $(E,\alpha,\beta)$ with $E\cong E_{t_2,t_3}$ is obtained by means of a change of basis matrix $S=\begin{psmallmatrix}s_1 & 0\\s_2 & s_3\end{psmallmatrix}$ such that 

\begin{equation}\label{defini}
\begin{pmatrix}\alpha \\ \beta\end{pmatrix} = \begin{pmatrix}s_1 & 0\\s_2 & s_3\end{pmatrix} \begin{pmatrix}\dfrac{dx}{y}\\ \dfrac{x\,dx}{y}\end{pmatrix}.
\end{equation}

Such an $S$ satisfies \begin{equation}\label{cond}S\begin{psmallmatrix}0 & 1\\  -1 & 0\end{psmallmatrix}S^T=\begin{psmallmatrix}0 & 1\\  -1 & 0\end{psmallmatrix}.\end{equation} Reciprocally, any $\alpha,\beta$ defined through equation \ref{defini}, by means of an $S$ satisfying condition \ref{cond}, satisfies Definition \ref{enhanced}. Observe now that condition \ref{cond} is equivalent to $s_1s_3=0$. Therefore, any such $S$ is of the form $\begin{psmallmatrix}s & 0\\  t & s^{-1}\end{psmallmatrix},$ with $s\in\mathbb{C}^{*}$.

Let $\Psi: \text{Spec}(\mathbb{C}[t_1,t_2,t_3,\frac{1}{27t_3^2-t_2^3}])\rightarrow \mathsf{T}$ be defined by $(t_1,t_2,t_3)\mapsto (E_{t_2,t_3},\begin{psmallmatrix}1 & 0\\  t_1 & 1\end{psmallmatrix}\begin{psmallmatrix}\frac{dx}{y} \\ \frac{xdx}{y}\end{psmallmatrix})$. We affirm that $\Psi$ is bijective.

To prove surjectivity, pbserve Lemma \ref{second} implies that any given enhanced elliptic curve $(E_{t_1,t_2,t_3},\begin{psmallmatrix}s & 0\\  t & s^{-1}\end{psmallmatrix}\begin{psmallmatrix}\frac{dx}{y} \\ \frac{xdx}{y}\end{psmallmatrix})$ is isomorphic to $(E_{s^4t_2,s^6t_3},\begin{psmallmatrix}1 & 0\\  ts^{-1} & 1\end{psmallmatrix}\begin{psmallmatrix}\frac{dx}{y} \\ \frac{xdx}{y}\end{psmallmatrix}).$ Next, we deal with injectivity. Let us assume that $(E_{t_2,t_3},\begin{psmallmatrix}1 & 0\\  t_1 & 1\end{psmallmatrix}\begin{psmallmatrix}\frac{dx}{y} \\ \frac{xdx}{y}\end{psmallmatrix})\cong
(E_{t_2',t_3'},\begin{psmallmatrix}1 & 0\\  t_1' & 1\end{psmallmatrix}\begin{psmallmatrix}\frac{dx}{y} \\ \frac{xdx}{y}\end{psmallmatrix})$. Then, $E_{t_2,t_3}\cong E_{t_2',t_3'}$, and, by Theorem \ref{weierstrass}, there exists $\kappa\in\mathbb{C}$ such that $t_i'=\kappa^it_i$ for $i=2,3.$ By Lemma \ref{second}, we have that $(E_{\kappa^2t_2,\kappa^3t_3},\begin{psmallmatrix}1 & 0\\  t_1' & 1\end{psmallmatrix}\begin{psmallmatrix}\frac{dx}{y} \\ \frac{xdx}{y}\end{psmallmatrix})\cong
(E_{t_2,t_3},\begin{psmallmatrix}1 & 0\\  t_1' & 1\end{psmallmatrix}\begin{psmallmatrix}\lambda\frac{dx}{y} \\ \lambda^{-1}\frac{xdx}{y}\end{psmallmatrix})$ for any square root $\lambda\in\mathbb{C}$ of $\kappa$. Therefore, we have that $(E_{t_2,t_3},\begin{psmallmatrix}1 & 0\\  t_1 & 1\end{psmallmatrix}\begin{psmallmatrix}\frac{dx}{y} \\ \frac{xdx}{y}\end{psmallmatrix})\cong (E_{t_2,t_3},\begin{psmallmatrix}1 & 0\\  t_1' & 1\end{psmallmatrix}\begin{psmallmatrix}\lambda\frac{dx}{y} \\ \lambda^{-1}\frac{xdx}{y}\end{psmallmatrix})$. Let $\phi$ be such an automorphism. By Lemma \ref{automorphism}, $\phi(x,y)=(\mu^2x,\mu^3y)$ for some $\mu\in \mathbb{C}$ with $\mu^{2i}t_i=t_i$. Then, $\lambda\frac{dx}{y}=\phi^*(\frac{dx}{y})=\mu^{-1}\frac{dx}{y}$ and $t_1'\lambda\frac{dx}{y}+\lambda^{-1}\frac{xdx}{y}=\phi^*(t_1\frac{dx}{y}+\frac{xdx}{y})=t_1\mu^{-1}\frac{dx}{y}+\mu\frac{xdx}{y}$. Therefore, $\lambda=\mu^{-1}$, and $t_1'\lambda=t_1\lambda$, which implies $t_1=t_1'$ and $t_i=\mu^{-2i}t_i=\lambda^{2i}t_i=\kappa^it_i=t_i'$ for $i=1,2$. This concludes the proof. 
\end{proof}

\begin{remark} Since $(E_{t_2,t_3},\begin{psmallmatrix}1 & 0\\  t_1 & 1\end{psmallmatrix}\begin{psmallmatrix}\frac{dx}{y} \\ \frac{xdx}{y}\end{psmallmatrix})\cong (E_{t_1,t_2,t_3},\frac{dx}{y},\frac{xdx}{y})$, where  $E_{t_1,t_2,t_3}$ is given in affine coordinates as the zero-set of the polynomial $y^2-(x-t_1)^3+t_2(x-t_1)+t_3$, we see that the family \begin{equation}(E_{t_1,t_2,t_3},\frac{dx}{y},\frac{xdx}{y})), t_1,t_2,t_2\in\mathbb{C}, \Delta\neq 0\end{equation}is an universal family of enhanced elliptic curves. 
\end{remark}\label{remark}
To obtain quasi-modular forms as functions on $\mathbb{H},$ we need to make some trascendental considerations. 

\begin{proposition}
Let $(E,\alpha,\beta)$ be an enhanced elliptic curve and $\delta,\epsilon$ a basis of $H_1(E,\mathbb{Z})$ with $\langle \delta,\epsilon\rangle=1$. Let $P:=\begin{psmallmatrix}\int_{\delta}\alpha & \int_{\delta}\beta\\\int_{\epsilon}\alpha & \int_{\epsilon}\beta\end{psmallmatrix}.$  Then, $\begin{psmallmatrix} 0 & 1 \\ -1& 0\end{psmallmatrix}=P^T\begin{psmallmatrix} 0 & 1 \\ -1& 0\end{psmallmatrix}P$, $P\in \mathsf{GL}_2(\mathbb{C})$ and $(P^1)^T\begin{psmallmatrix} 0 & 1 \\ -1& 0\end{psmallmatrix}\overline{P^1}>0$.
\end{proposition}

\begin{proof}
First, we proof that $\begin{psmallmatrix} 0 & 1 \\ -1& 0\end{psmallmatrix}=P^{T}\begin{psmallmatrix} 0 & 1 \\ -1& 0\end{psmallmatrix} P$. Since $\begin{psmallmatrix} \alpha \\ \beta\end{psmallmatrix}=P^T\begin{psmallmatrix} 0 & 1 \\ -1& 0\end{psmallmatrix}\begin{psmallmatrix} \delta \\ \epsilon\end{psmallmatrix}^{pd},$ then $\begin{psmallmatrix} 0 & 1 \\ -1& 0\end{psmallmatrix}=\alpha\alpha^T=P^T\begin{psmallmatrix} 0 & 1 \\ -1& 0\end{psmallmatrix}\delta^{pd}(\delta^{pd})^T\begin{psmallmatrix} 0 & 1 \\ -1& 0\end{psmallmatrix}^TP=P^T\begin{psmallmatrix} 0 & 1 \\ -1& 0\end{psmallmatrix}\begin{psmallmatrix} 0 & 1 \\ -1& 0\end{psmallmatrix}\begin{psmallmatrix} 0 & 1 \\ -1& 0\end{psmallmatrix}^{T}P=P^T \begin{psmallmatrix} 0 & 1 \\ -1& 0\end{psmallmatrix}P.$ Observe that, since $det\begin{psmallmatrix} 0 & 1 \\ -1& 0\end{psmallmatrix}\neq 0$, this implies that $det(P)\neq 0$. Finally, the last assertion in this proposition follows from the Riemann bilinear relations.
\end{proof}

The previous proposition suggests the following definition:
\begin{definition}
The \textbf{manifold of period matrices} is 
\begin{equation*}\mathsf{\Pi}=\Big\{P\in \mathsf{GL}_2(\mathbb{C})|\, \begin{psmallmatrix} 0 & 1 \\ -1& 0\end{psmallmatrix}=P^{T}\begin{psmallmatrix} 0 & 1 \\ -1& 0\end{psmallmatrix} P \text{ and } (P^1)^T\begin{psmallmatrix} 0 & 1 \\ -1& 0\end{psmallmatrix}\overline{P^1}>0\Big\}.
\end{equation*}
\end{definition}

\begin{remark}
By using standard differential geometric arguments, it can be proved that $\mathsf{\Pi}$ is a smooth complex manifold of dimension $3$, and that for any $P\in \mathsf{\Pi},$ $T_P\mathsf{\Pi}=\{X\in\mathsf{Mat}_2(\mathbb{C})|\, P^T\begin{psmallmatrix} 0 & 1 \\ -1& 0\end{psmallmatrix} X+X^T\begin{psmallmatrix} 0 & 1 \\ -1& 0\end{psmallmatrix} P=0\}$.
\end{remark}

\begin{definition}
The \textbf{generalized period map} is the map $\mathcal{P}:\mathsf{T}\rightarrow \mathsf{SL}_2(\mathbb{Z})\backslash \mathsf{\Pi}$ which around a given $t=(E_t,\alpha_t,\beta_t)\in\mathsf{T}$ is defined by considering a local continuous family homology classes $\delta(t),\epsilon(t)$ basis of $H_1(E_t,\mathbb{Z})$ with $\langle \delta(t),\epsilon(t)\rangle=1$, and making $$\mathcal{P}(t)=\begin{psmallmatrix}\int_{\delta(t)}\alpha(t) & \int_{\delta(t)}\beta(t)\\\int_{\epsilon(t)}\alpha(t) & \int_{\epsilon(t)}\beta(t)\end{psmallmatrix}.$$
\end{definition}

\begin{theorem}
$\mathcal{P}$ is a biholomorphism.
\end{theorem}
\begin{proof}
This is an easy consequence of the Torelli theorem for curves,
\end{proof}

\begin{definition} We define the complex algebraic group
\begin{equation*}
\mathsf{G}=\{g\in \mathsf{Mat}_{2}(\mathbb{C})\, | \, g^T\begin{psmallmatrix} 0 & 1 \\ -1& 0\end{psmallmatrix} g=\begin{psmallmatrix} 0 & 1 \\ -1& 0\end{psmallmatrix}\land g_{21}=0 \}.
\end{equation*}
\end{definition}
\begin{remark}
Observe that $\mathsf{G}$ acts on the right of $\mathsf{T}$ by means of $(E,\begin{psmallmatrix}\alpha\\\beta\end{psmallmatrix})\cdot g= (E,g^T\begin{psmallmatrix}\alpha\\\beta\end{psmallmatrix})$, and $\mathsf{T}/\mathsf{G}$ is isomorphic to the moduli space $\mathcal{M}$ of complex elliptic curves.
\end{remark}

\begin{definition}
The AMSY Lie algebra $\mathfrak{G}$ associated to elliptic curves is by definition the Lie sub algebra of $\mathfrak{gl}_2(\mathbb{C})$ generated by $\text{Lie}(\mathsf{G})$ and $x^T$ for every $x$ in the nilradical of $\text{Lie}(\mathsf{G})$.
\end{definition}

\begin{remark}
The previous algebra has also been called the Gauss-Manin Lie algebra in the references.
\end{remark}

\begin{proposition}
$\text{Lie}(\mathsf{G})$ is the Lie subalgebra of $\mathfrak{gl}_2(\mathbb{C})$ freely generated by $\mathfrak{g}_0=\begin{psmallmatrix} 1 & 0 \\ 0& -1\end{psmallmatrix}$ and $\mathfrak{g_1}=\begin{psmallmatrix} 0 & 1 \\ 0& 0\end{psmallmatrix}$. 
\end{proposition}

\begin{proof}
This follows by observing that $\text{Lie}(\mathsf{G})=\{x\in \mathsf{Mat}_{5}(\mathbb{C})|\,x^T \begin{psmallmatrix} 1 & 0 \\ 0& -1\end{psmallmatrix}+\begin{psmallmatrix} 1 & 0 \\ 0& -1\end{psmallmatrix}x=0\land x_{21}=0\}$. 
\end{proof}

\begin{corollary}$\mathfrak{G}$ is isomorphic to $\mathfrak{sl}_2(\mathbb{C})$.
\end{corollary}
\begin{lemma} For each $\mathfrak{g}\in \mathfrak{G}$, there is a unique vector field $\tilde{\mathsf{R}}_{\mathfrak{g}}\in\Theta_{\text{SL}_2(\mathbb{Z})\backslash\mathsf{\Pi}}$ such that
\begin{equation}\label{elip}
[dx_{ij}(\tilde{\mathsf{R}}_{\mathfrak{g}})]=[x_{ij}]\mathfrak{g}.
\end{equation}
\end{lemma}

\begin{proof}
Let $x_{ij}$ be the usual matrix coordinates for $\mathsf{Mat}_2(\mathbb{C})$, and let us define 
\begin{equation}
\tilde{\mathsf{R}}_{\mathfrak{g}}=\sum_{i,j=1}^2 c_{ij}\frac{\partial}{\partial x_{ij}}\in \Theta_{\mathsf{Mat}_2(\mathbb{C})}
\end{equation}
by means of equation \ref{elip}. Therefore, $[c_{ij}]=[x_{ij}]\mathfrak{g}.$ The previous corollary implies that, for each $P\in \mathsf{\Pi}$, $P^T\begin{psmallmatrix} 0 & 1 \\ -1& 0\end{psmallmatrix}(\mathsf{\tilde{R}}_{\mathfrak{g}})_P+(\mathsf{\tilde{R}}_{\mathfrak{g}})_P^T\begin{psmallmatrix} 0 & 1 \\ -1& 0\end{psmallmatrix} P=P^T\begin{psmallmatrix} 0 & 1 \\ -1& 0\end{psmallmatrix} P \mathfrak{g}+\mathfrak{g}^T P^T\begin{psmallmatrix} 0 & 1 \\ -1& 0\end{psmallmatrix} P=\begin{psmallmatrix} 0 & 1 \\ -1& 0\end{psmallmatrix} \mathfrak{g}+\mathfrak{g}^T\begin{psmallmatrix} 0 & 1 \\ -1& 0\end{psmallmatrix}=0$. Therefore, $\tilde{\mathsf{R}}_{\mathfrak{g}}$ descends to a holomorphic vector field on $\mathsf{\Pi}$, which we also denote by $\tilde{\mathsf{R}}_{\mathfrak{g}}\in\Theta_{\mathsf{\Pi}}$. Since the action of $\text{SL}_2(\mathbb{Z})$ on $\mathsf{\Pi}$ is given by left multiplication, its derivative is also given by left multiplication. Therefore, equation \ref{elip} implies that the holomorphic vector fields  $\tilde{\mathsf{R}}_{\mathfrak{g}}$ are $\text{SL}_2(\mathbb{Z})$-invariant, which concludes the proof.

\end{proof}

\begin{theorem}\label{bigteo}
For each $\mathfrak{g}\in \mathfrak{G}$, there exists a unique algebraic vector field $\mathsf{R}_{\mathfrak{g}}\in \Theta_{\mathsf{T}}$ such that
\begin{equation}\label{equation}
\nabla_{\mathsf{R}_{\mathfrak{g}}} \begin{psmallmatrix}\alpha\\
\beta\end{psmallmatrix}=\mathfrak{g}^T\begin{psmallmatrix}\alpha\\
\beta\end{psmallmatrix}.
\end{equation}
Here, $\begin{psmallmatrix}\alpha\\
\beta\end{psmallmatrix}=\begin{psmallmatrix}1 & 0\\  t_1 & 1\end{psmallmatrix}\begin{psmallmatrix}\frac{dx}{y} \\ \frac{xdx}{y}\end{psmallmatrix}$ as in Remark \ref{remark}.

\end{theorem}

\begin{proof}
We are going to explicitly compute the vector fields in the statement of this theorem. Let us fix an arbitrary $\mathfrak{g}\in \mathfrak{G}$, and write $\mathsf{R}=\mathsf{R}_{\mathfrak{g}}$ and $S=\begin{psmallmatrix}1 & 0\\  t_1 & 1\end{psmallmatrix}$. We begin by observing that, by the previous lemma, the existence of an unique holomorphic vector field $\mathsf{R}$ on $\mathsf{T}$ satisfying Equation \ref{equation} is already guaranteed. To check that it is algebraic, we use a linear algebra argument. Let 

\begin{equation}
\mathsf{R}=a_1\frac{\partial}{\partial t_1}+ a_2\frac{\partial}{\partial t_2}+a_3\frac{\partial}{\partial t_3}, 
\end{equation} 
where $a_k$ are holomorphic functions on $\mathsf{T}$. Now, we aim to prove that these functions are indeed regular.

First, a calculation using tame polynomials theory (see Chapter 3) yields

\begin{equation}
\mathsf{GM}_{\begin{psmallmatrix}\frac{dx}{y}\\ \frac{xdx}{y}\end{psmallmatrix}}=
\frac{1}{\Delta}\begin{bmatrix}
-\frac{1}{12}d\Delta & \frac{3}{2} \alpha\\
-\frac{1}{8}t_2\alpha & \frac{1}{12}d\Delta 
\end{bmatrix}, \alpha=3t_3dt_2-2t_2dt_3.
\end{equation}

Since $\mathsf{GM}_{\
\begin{psmallmatrix}\alpha\\\beta
\end{psmallmatrix}
}=(dS+S\mathsf{GM}_{\begin{psmallmatrix}\frac{dx}{y}\\ \frac{xdx}{y}\end{psmallmatrix}})S^{-1}$, Equation \ref{equation} is equivalent to

\begin{equation}\label{despejar}
\begin{bmatrix}
0 & 0\\
a_1 & 0
\end{bmatrix}
+\frac{1}{\Delta}\begin{bmatrix}
1 & 0\\
t_1 & 1
\end{bmatrix}
\begin{bmatrix}
-\frac{1}{12}(54t_3a_3-3t_2^2a_2) & \frac{3}{2}(3t_3a_2-2t_2a_3)\\
-\frac{1}{8}t_2(3t_3a_2-2t_2a_3) & \frac{1}{12}(54t_3a_3-3t_2^2a_2)
\end{bmatrix}=dS+S\mathsf{GM}_{\begin{psmallmatrix}\frac{dx}{y}\\ \frac{xdx}{y}\end{psmallmatrix}}=\mathfrak{g}^TS.
\end{equation}

Entries (1,1) and (1,2) of the previous equality allow us to produce the following linear system
\begin{equation}
\frac{1}{\Delta}\begin{bmatrix}
\frac{1}{4}t_2^2 & -\frac{27}{6}t_3\\
\frac{9}{2}t_3 & -3t_2
\end{bmatrix}\begin{bmatrix}
a_2 \\
a_3
\end{bmatrix}=\begin{bmatrix}
(\mathfrak{g}^TS)_{11}\\
(\mathfrak{g}^TS)_{12}.
\end{bmatrix}
\end{equation}
It can be explicitly inverted to give us 

\begin{equation}
\begin{bmatrix}
a_2 \\
a_3
\end{bmatrix}=\frac{4}{3}\begin{bmatrix}
-3t_2 & \frac{27}{6}t_3\\
-\frac{9}{2}t_3 & \frac{1}{4}t_2^2
\end{bmatrix}\begin{bmatrix}
(\mathfrak{g}^TS)_{11}\\
(\mathfrak{g}^TS)_{12}
\end{bmatrix}.
\end{equation}

Furthermore, entry (2,1) in Equation \ref{despejar}, gives us 

\begin{equation}a_1=(\mathfrak{g}^TS)_{21}-\frac{1}{\Delta}(-\frac{1}{12}t_1(54t_3a_3-3t_2^2a_2)-\frac{1}{8}t_2(3t_3a_2-2t_2a_3)).
\end{equation}

Therefore, $a_1,a_2,a_3$ are global regular functions on $\mathsf{T}$.
\end{proof}

\begin{remark}\label{computation}
The previous proof allows us to explicitly compute:

\begin{flalign}
\mathsf{R}_{\mathfrak{g}_1^T} &= \left(\frac{1}{12}t_2 - t_1^2\right)\frac{\partial}{\partial t_1} + (6t_3 - 4t_1 t_2)\frac{\partial}{\partial t_2} + \left(\frac{1}{3}t_2^2 - 6t_1 t_3\right)\frac{\partial}{\partial t_3}, \\
\mathsf{R}_{\mathfrak{g}_0} &= -2t_1\frac{\partial}{\partial t_1} - 4t_2\frac{\partial}{\partial t_2} - 6t_3\frac{\partial}{\partial t_3}, \\
\mathsf{R}_{\mathfrak{g}_1} &= \frac{\partial}{\partial t_1}.
\end{flalign}

The $\mathcal{O}_{\mathsf{T}}$-module generated by the previous vector fields is also called the AMSY Lie algebra or the Gauss-Manin Lie algebra associated to enhanced elliptic curves. 

Observe now that, since the Gauss-Manin connection $\nabla$ is flat, Equation \ref{equation} implies that 

\begin{equation}
[\mathsf{R}_{\mathfrak{g}},\mathsf{R}_{\mathfrak{g}'}]=\mathsf{R}_{[\mathfrak{g'},\mathfrak{g}]^T}.
\end{equation}

This implies that the previous vector fields form an $\mathfrak{sl}_2(\mathbb{C})$-algebra.

\end{remark}

\begin{definition}
The mirror map, or $\tau$-map is the map $$\tau:\mathbb{H}\rightarrow \mathsf{\Pi},$$
$$z\mapsto \begin{psmallmatrix} z & 1 \\ -1 & 0\end{psmallmatrix}.$$
\end{definition}
In our context, the mirror map allows us to reconstruct from the complex structure of an elliptic curve, its whole Hodge structure.
\begin{definition}
The trascendental map or $\mathsf{t}$-map is the composition $$\mathbb{H}\xrightarrow\tau\mathsf{\Pi}\rightarrow \mathsf{SL}_2(\mathbb{Z})\backslash \mathsf{\Pi}\xrightarrow {\mathcal{P}^{-1}}\mathsf{T}$$.
\end{definition}

By means of the $\mathsf{t}$-map we can pullback regular functions on $\mathsf{T}$ to holomorphic functions on $\mathbb{H}$.

\begin{definition}
Let $g_i:=\mathsf{t}^*(t_i):\mathbb{H}\rightarrow\mathbb{C}$ for $i=1,2,3$.
\end{definition}

\begin{theorem}
 $$g_1'=g_1^2-\frac{1}{12}g_2,$$
    $$g_2'=4g_1g_2-6g_3,$$
    $$g_3'=6g_1g_3-\frac{1}{3}g_2^2.$$
\end{theorem}
\begin{proof}
This follows from Remark \ref{computation}, where the Ramanujan vector field $\mathsf{R}_{\mathfrak{g}_1^T}$ is computed.
\end{proof}

\begin{definition}
Let $$E_2=12g_1,\,\, E_4=12g_2,\,\, E_6=(12)(18)g_3.$$
\end{definition}

\begin{theorem}
These functions coincide with the Eisenstein series previously defined.
\end{theorem}
\begin{proof}
One way to do this is to use the previous theorem and Theorem \ref{ramanujan}, after comparing the first terms in their $q$-expansions.
\end{proof}

\begin{corollary}The graded algebra of quasi-modular forms $\mathfrak{\tilde{M}}$ coincides with the pullback under $\mathsf{t}$ of the graded algebra $\mathbb{C}[t_1,t_2,t_2]$ with $deg(t_i)=2i$.
\end{corollary}

\section{Mirror symmetry for elliptic curves}

\begin{definition} 
Let $E$ be an elliptic curve over $\mathbb{C}$, and $g\geq 1, m\geq 2$ integers. A pair $(C,f)$ consisting of a smooth complex curve $f$ and a holomorphic map $f:C\rightarrow E$ is called an \textit{simple $m$-branched cover} if $C$ is connected, every ramification point has index $m$, and different ramification points have different images.

\end{definition}

Let $X_{g,d}^{(m)}$ be the set of isomorphism classes of $m$-branched covers of $E$ of degree $d$ and genus $g$, and let $$N_{g,d}^{(m)}=\sum_{(C,f)\in X_{g,d}^{(m)}} \frac{1}{|Aut{(C,f)}|}.$$

Define $$F_g^{(m)}(q)=\sum_{d\geq 1}N_{g,d}^{(m)}q^d,$$

for $g\geq 2$. The case $g=1$ must be treated separately, since we have to consider the contribution of constant maps. We define $$F_1^{(m)}(q)=-\frac{1}{24} \log q+\sum_{d\geq 1}N_{1,d}^{(m)}q^d.$$

In this document will present the proof the case $m=2$ of the next theorem; for the general case
we refer to \cite{Och}.

\begin{theorem}\parencite{dijkgraaf,ochiai} For $g\geq 2$, $F_g^{(m)}(q)\in \mathfrak{\tilde{M}}$. Furthermore, $F_g^{(2)}(q)\in \mathfrak{\tilde{M}}_{6g-g}.$
\end{theorem}

\subsection{Partition functions}

One of the steps in the proof of theorem 1 is to allow possibly disconnected covers of $E$, so let $\hat{X}_{g,d}$ and $\hat{N}_{g,d}$ be defined as in the introduction, but allowing possibly disconnected smooth curves $C$. $\hat{N}_{g,d}$ and $N_{g,d}$ are finite by proposition 2 in the next section. Now we define partition functions for the connected and possibly disconnected case, respectively, by $$Z(q,\lambda)=\sum_{g\geq 1}\sum_{d\geq 1}\frac{N_{g,d}}{(2g-2)!}q^d\lambda^{2g-2}=\sum_{g\geq 1}\frac{F_g(q)}{(2g-2)!}\lambda^{2g-2},$$

and 

$$\hat{Z}(q,\lambda)=\sum_{g\geq 1}\sum_{d\geq 1}\frac{\hat{N}_{g,d}}{(2g-2)!}q^d\lambda^{2g-2}.$$

The strategy is to relate this two partitions functions by means of the next proposition.

\begin{proposition}
$\hat{Z}(q,\lambda)=\exp(Z(q,\lambda))-1$.
\end{proposition}

The proof of it is stricly combinatorial and can be considered as a consequence of the exponential formula in combinatorics. See, for example, Corollary 3.4.1. in \cite{wilf}.

\subsection{Reducing $\hat{N}_{g,d}$ to the trace on an operator}

Let us choose $b$ different points $P_1,\ldots, P_b$ in $E$, where $b=2g-2$, that is, by Hurwitz formula, the number of ramification points of a $2$-branched cover of $E$, and let $\pi_1^b$ be the fundamental group of the $b$-punctured elliptic curve $E-{P_1,\ldots,P_b}$. Then $$\pi_1^b=\langle \gamma_1,\ldots,\gamma_b,\alpha,\beta \, |\,  \gamma_1\cdots \gamma_b=\alpha\beta\alpha^{-1}\beta^{-1}\rangle$$.

Let $c_{trans}$ be the conjugation class of the transpositions in $S^d$ . Define $$\Phi_{g,d}=\{\phi\in Hom(\pi_1^b,S_d)\, |\, \phi(\gamma_i)\in c_{trans}\text{ for every }i\}.$$ Then $S_d$ acts on $\Phi_{g,d}$ by conjugation.

By the next proposition, computing $\hat{N}_{g,d}$ is equivalent to computing $|\Phi_{g,d}|$. 

\begin{proposition}
\begin{itemize}
\item[a.] Monodromy induces a bijection $\hat{X}_{g,d}\cong \Phi_{g,d}/S_d$;
\item[b.] $\hat{N}_{g,d}=|\Phi_{g,d}|/d!$.
\end{itemize}
\end{proposition}

We want now want to compute $|\Phi_{g,d}|$. Since we have a description of $\pi_1^b$ by generators and relations, we get a bijective correspondence between $\Phi_{g,d}$ and the set $$\{(\tau_1,\ldots,\tau_b,\tau,\sigma)\in (S_d)^{b+2}\, |\, \tau_1\cdots \tau_b=\tau\sigma\tau^{-1}\sigma^{-1},\, \tau_i\in c_{trans}\text{ for every }i\}.$$ 
So we need to count the elements of this set. Fix $\sigma\in S_d$. Let $$P_{g,d}(\sigma)=\{(\tau_1,\ldots,\tau_b)\in (S_d)^b\, |\, \tau_1\cdots \tau_b\sigma\text{ and } \sigma \text{ are conjugate}, \tau_i\in {c_{trans}}\text{ for every }i\}.$$ 

Since given a $\tau$ such that  $\tau_1\cdots \tau_b=\tau\sigma\tau^{-1}\sigma^{-1}$ we can get all the other $\tau$ with such a property by multiplying it with an element of the conjugation class of $\sigma$, $c(\sigma)$, we have $$|\Phi_{g,d}|=\sum_{{\sigma\in{S_d}}}\frac{d!}{|c(\sigma)|}|P_{g,d}(\sigma)|.$$ 

Since $P_{g,d}(\tau\sigma\tau^{-1})=\tau P_{g,d}(\sigma)\tau^{-1}$, $|P_{g,d}|:S_d \rightarrow \mathbb{N}$ is a class function, therefore $$|\Phi_{g,d}|=\sum_{{\sigma\in{\mathcal{C}_d}}} d!|P_{g,d}(\sigma)|,$$ and $\hat{N}_{g,d}=\sum_{c\in{\mathcal{C}_d}} |P_{g,d}(c)|$. 

For now on, let's fix a complete set of representatives $\{\sigma_1,\ldots,\sigma_m\}$ of the conjugation classes $\mathcal{C}_d=\{c_1,\ldots,c_m\}$, $c(\sigma_i)=c_i$. Observe that $m$ is the number of partitions of $d$ (that we call $part(d)$) since the conjugacy class of an element of $S_d$ is determined by the partition induced by its cycle decomposition Consider the matrix $$(M_d)_{ij}=|\{\tau\in c_{trans}\, |\, \tau\sigma_i\in c_j\}|.$$

An straightforward induction argument, with the base case being the last definition gives us $$(M_d^k)_{ij}=|\{(\tau_1,\ldots,\tau_k)\in (S_d)^k\, | \, \tau_1,\cdots\tau_k\sigma_i\in c_j, \tau_l\in c_{trans}\text{ for every }l\}|.$$ 

From this observation we get $(M_d^b)_{i,i}=|P_{g,d}(\sigma_i)|$ and therefore we have $\hat{N}_{g,d}=\operatorname{Tr}(M_d^{b})$. We also conclude that $(M_d^k)_{i,i}=0$ for odd $k$. Therefore we have the following expression for $\hat{Z}$:
\[
\hat{Z}(q,\lambda)=\sum_{g\geq 1}\sum_{d\geq 1}\frac{\operatorname{Tr}(M_d^{b})}{(2g-2)!}q^d\lambda^{2g-2}=\sum_{d\geq 1}\operatorname{tr}(\exp(M_d\cdot \lambda))q^d.
\]

Let $\alpha_{d,i}$ be the eigenvalues of $M_d$. Then $\hat{N}_{g,d}=\sum_{i=1}^m \alpha_{d,i}^k$. So the last expression gives us the following proposition.

\begin{proposition}
$.\hat{Z}(q,\lambda)=\sum_{d\geq 1}\sum_{i\geq 1}^{part(d)}\exp(\alpha_{d,i}\lambda)q^d$.
\end{proposition}

\subsection{Diagonalizing $M_d$}

Recall that the class algebra $\mathcal{H}_d$, which is the center of the group algebra $\mathbb{C}[S_d]$  (i.e., functions $S_d\rightarrow \mathbb{C}$ constant on conjugation classes) has a basis $(z_c)_{c\in \mathcal{C}_d}$, where $z_c=\sum_{\sigma\in c}\sigma$. Since the inverse of a transposition is a transposition, we have the following easy proposition.

\begin{proposition}
The matrix of the operator on $\mathcal{H}_d$ defined by multiplication by $z_{c_{trans}}$ is the transpose of $M_d$.
\end{proposition}

Now, the eigenvalues of this operator follow from Schur orthogonality relations: 

For any irreducible character $\chi$ of $S_d$, define $$w_{\chi}=\frac{dim(\chi)}{d!}\sum_{c\in{\mathcal{C}_d}}\chi(c^{-1})z_c=\frac{dim(\chi)}{d!}\sum_{\sigma\in{\mathcal{C}_d}}\chi(\sigma^{-1})\sigma.$$

Schur orthogonality relations tell us that $w_{\chi}\cdot w_{\chi'}=\delta_{\chi,\chi'}w_{\chi}$ (this is Kronecker's delta), and $$z_c=\sum_{\chi}=\big(\frac{|c^{-1}|\chi(c^{-1})}{dim(\chi)}\big)w_\chi.$$ From this identities it follows immediately that $(w_\chi)_{\chi\in \mathcal{R}_d}$ is basis of $\mathcal{H}_d$ which diagonalizes multiplication by $z_{c_{trans}}$. The eigenvector $w_{\chi}$ has eigenvalue $$\frac{{d\choose 2} \chi(c)}{dim(\chi)}.$$

\subsection{Frobenius formula and Haneko-Zagier result}

It is known that irreducible representations of $S_d$ are in one-to-one correspondence with partitions of $d$, (for example, via the Sprecht representation). The particular construction of this representations will not be relevant to us, but a fomula of Frobenius will. Let $\chi$ be an irreducible representation of $S_d$ and $\lambda=(\lambda_1,\ldots, \lambda_l)$ its corresponding partition. Define $u_i=\lambda_i-i+\frac{1}{2}$ and $v_i$ defined as for $u_i$, but by considering the transpose partition of $\lambda$.. The Frobenius formula we will use affirms that $$\frac{{d\choose 2} \chi(c)}{dim(\chi)}=\frac{1}{2}\sum_i (u_i^2-v_i^2),$$ so it computes the eigenvalues found in the last section.

Consider the product $$\prod_{u\in \mathbb{N}+\frac{1}{2}}(1+\zeta q^u e^{\frac{u^2}{2}\lambda})\prod_{v\in \mathbb{N}+\frac{1}{2}}(1+\zeta^{-1} q^u e^{\frac{-v^2}{2}\lambda})$$

We can see it as a Laurent formal series in $\zeta$. Propositions 3 and 4 gives us 

\begin{proposition} $\hat{Z}(q,\lambda)=\text{coef. of $\zeta^0$ in} \prod_{u\in \mathbb{N}+\frac{1}{2}}(1+\zeta q^u e^{\frac{u^2}{2}\lambda})\prod_{v\in \mathbb{N}+\frac{1}{2}}(1+\zeta^{-1} q^u e^{\frac{-v^2}{2}\lambda})-1.$ \end{proposition}

In \cite{Kaneko}, Kaneko and Zagier considered the series
\[
\Theta(q,\lambda,z)=\prod_{n\geq 1}(1-q^n)\prod_{u\in \mathbb{N}+\frac{1}{2}}\bigl(1+\zeta q^u e^{\frac{u^2}{2}\lambda}\bigr)\prod_{v\in \mathbb{N}+\frac{1}{2}}\bigl(1+\zeta^{-1} q^v e^{-\frac{v^2}{2}\lambda}\bigr).
\]
Let $\Theta_0(q,\lambda)$ be the coefficient of $\zeta^0$ in $\Theta(q,\lambda,z)$. Then, since $\Theta_0(q,\lambda)=\Theta_0(q,-\lambda)$, we have an expansion
\[
\Theta_0(q,\lambda)=\sum_{n\geq 0} A_n(q)\lambda^{2n}.
\]
Kaneko and Zagier prove (Theorem~1, \cite{Kaneko}) that $A_n(q)$ is a quasimodular form of weight $6n$. This implies that the coefficient of $\lambda^{2g-2}$ in $\log  \Theta_0(q,\lambda)$ is a modular form of weight $6g-6$, since, by using the Taylor series of the logarithm around $-1$, we have
\[
\log \Theta_0(q,\lambda)=\log\!\Bigl(1+\sum_{n\geq 1} A_n(q)\lambda^{2n}\Bigr)=\sum_{m\geq 1}\frac{(-1)^{m+1}}{m}\Bigl(\sum_{n\geq 1} A_n(q)\lambda^{2n}\Bigr)^{\!m}.
\]
On the other hand, by Propositions~1 and~5, we get
\[
\log \Theta_0(q,\lambda)=\log\!\Bigl(\prod_{n\geq 1}(1-q^n)\bigl(\hat{Z}(q,\lambda)+1\bigr)\Bigr)=\sum_{n\geq 1}\log(1-q^n)+\log\!\bigl(\hat{Z}(q,\lambda)+1\bigr)=\sum_{n\geq 1}\log(1-q^n)+Z(q,\lambda).
\]
Since the coefficient of $\lambda^{2g-2}$ in this last sum is the one coming from $Z(q,\lambda)$, i.e.,
\[
\frac{F_g(q)}{(2g-2)!},
\]
the proof of Theorem~1 for $m=2$ is concluded.

\printbibliography

@string{A = {American Institute of Mathematical Sciences}}

@string{I = {Israel Journal of Mathematics}}

@article{pascal,
    author        = {Hossein Movasati},
    title         = {Quasi-modular forms attached to elliptic curves, I},
    journaltitle  = {
Annales Mathématiques Blaise Pascal},
    year          = {2012}

    
}

@article{mov12,
    author        = {Hossein Movasati},
    title         = {Quasi-modular forms attached to elliptic curves, I},
    journaltitle  = {Annales Mathématiques Blaise Pascal},
    year          = {2012}

    
}

@book{123,
    author        = {D. Zagier and J. H. Bruinier and G. van der Geer and G, Harder},
    title         = {The 1-2-3 of modular forms},
    year          = {2008}
}

@article{kaneko,
    author        = {M. Kaneko and D. Zagier},
    title         = {A Generalized Jacobi Theta Function and Quasimodular Forms},
    journaltitle={Progress in Math},
    year          = {1995}
    
}

@article{dijkgraaf,
    author        = {E. Dijkgraaf},
    title         = {Mirror Symmetry and Elliptic curves},
    journaltitle={Prog. Math},
    year= {1995}
    
}

@article{ochiai,
    author        = {H. Ochiai},
    title         = {Counting functions for branched covers and quasimodular forms},
    journaltitle={RIMS Kokyuroko},
    year = {2001}
    
}

@article{ramanujan,
    author        = {S. Ramanujan},
    title         = {On certain arithmetical functions},
    journaltitle={Trans. Cambridge Philos. Soc.},
    year = {1916}
    
}

@article{mov2008,
    author        = {Hossein Movasati},
    title         = {Differential modular forms and some analytic relations between Eisenstein series},
    journaltitle={Ramanujan Journal},
    year = {2008}
    
}

@article{mov2012,
    author        = {Hossein Movasati},
    title         = {On Ramanujan relations between Eisenstein series},
    journaltitle={Manuscripta Mathematicae},
    year = {2012}
    
}

@book{serre,
    author        = {Jean-Pierre Serre},
    title         = {A Course in Arithmetic},
    year          = {1973},
    publisher ={Springer-Verlag new York Inc.}
}

@incollection {Och,
    AUTHOR = {Ochiai, H.},
     TITLE = {Counting functions for branched covers of elliptic curves and quasi modular forms},
 BOOKTITLE = {Representation theory of vertex operator algebras and related topics (Japanese)},
    SERIES = {},
    VOLUME = {1218},
     PAGES = {153–167.},
 PUBLISHER = {},
      YEAR = {2001},
   MRCLASS = {},
  MRNUMBER = {},
MRREVIEWER = {},
       DOI = {},
       URL = {},
}

@book{diamond,
  title = {A First Course in Modular Forms},
  author = {Fred Diamond and Jerry Shurman},
  publisher = {Springer},
  year = {2005},
  series = {Graduate Texts in Mathematics},
  volume = {228},
  isbn = {038723229X},
  address = {New York}
}

@book{koblitz1993,
  title = {Introduction to Elliptic Curves and Modular Forms},
  author = {Neal Koblitz},
  publisher = {Springer},
  year = {1993},
  series = {Graduate Texts in Mathematics},
  volume = {97},
  isbn = {0387946450},
  address = {New York}
}

@book{wilf,
  author = {Wilf, Herbert S.},
  title = {Generatingfunctionology},
  publisher = {Academic Press},
  year = {1990},
  isbn = {0-12-751955-6},
  address = {Boston},
  pages = {184}
}

@article{Paez2025,
  title        = {Modular Vector Fields for Lattice Polarized K3},
  author       = {Páez Gaviria, Walter Andr\'es},
  journal      = {Journal of Algebra},
  volume       = {673},
  pages        = {410--454},
  year         = {2025},
  doi          = {10.1016/j.jalgebra.2025.03.002},
  url          = {https://www.sciencedirect.com/science/article/pii/S0021869325001139}
}

@article{Movasati2015Hecke,
  author       = {Movasati, Hossein},
  title        = {Quasi‐modular forms attached to elliptic curves: Hecke operators},
  journal      = {Journal of Number Theory},
  volume       = {157},
  pages        = {424--441},
  year         = {2015},
  doi          = {10.1016/j.jnt.2015.05.021},
  url          = {https://www.sciencedirect.com/science/article/pii/S0022314X15001373}
}

@article{AlimMovasatiScheideggerYau2016,
  author       = {Alim, Murad and Movasati, Hossein and Scheidegger, Emanuel and Yau, Shing-Tung},
  title        = {Gauss–Manin connection in disguise: Calabi–Yau threefolds},
  journal      = {Communications in Mathematical Physics},
  volume       = {334},
  number       = {3},
  pages        = {889--914},
  year         = {2016},
  doi          = {10.1007/s00220-015-2510-3},
  eprint       = {arXiv:1410.1889},
  archivePrefix= {arXiv},
  primaryClass = {math.AG},
  url          = {https://arxiv.org/abs/1410.1889}
}

@article{Movasati2017NoetherLefschetz,
  author       = {Movasati, Hossein},
  title        = {Gauss–Manin connection in disguise: Noether‐Lefschetz and Hodge loci},
  journal      = {Asian Journal of Mathematics},
  volume       = {21},
  number       = {3},
  pages        = {463--482},
  year         = {2017},
  publisher    = {International Press},
  url          = {https://archive.intlpress.com/site/pub/files/_fulltext/journals/ajm/2017/0021/0003/AJM-2017-0021-0003-a003.pdf}
}

@article{HaghighatMovasatiYau2017,
  author       = {Haghighat, Babak and Movasati, Hossein and Yau, Shing-Tung},
  title        = {Calabi-Yau modular forms in limit: Elliptic Fibrations},
  journal      = {Communications in Number Theory and Physics},
  volume       = {11},
  number       = {4},
  pages        = {879--912},
  year         = {2017},
  doi          = {10.4310/CNTP.2017.v11.n4.a4},
  eprint       = {arXiv:1511.01310},
  archivePrefix= {arXiv},
  primaryClass = {math.AG},
  url          = {https://arxiv.org/abs/1511.01310}
}

@article{MovasatiNikdelan2021Dwork,
  author       = {Movasati, Hossein and Nikdelan, Younes},
  title        = {Gauss–Manin connection in disguise: Dwork Family},
  journal      = {Journal of Differential Geometry},
  volume       = {119},
  number       = {1},
  pages        = {73--98},
  year         = {2021},
  doi          = {10.4310/jdg/1631124264},
  url          = {https://projecteuclid.org/journals/journal-of-differential-geometry/volume-119/issue-1/GaussManin-connection-in-disguise-Dwork-family/10.4310/jdg/1631124264.full}
}

@article{CaoMovasatiVillaflorLoyola2024,
  author       = {Cao, Jin and Movasati, Hossein and Villaflor Loyola, Roberto},
  title        = {Gauss–Manin Connection in Disguise: Quasi Jacobi Forms of Index Zero},
  journal      = {International Mathematics Research Notices},
  volume       = {2024},
  number       = {8},
  pages        = {6680--6709},
  year         = {2024},
  doi          = {10.1093/imrn/rnad260},
  url          = {https://academic.oup.com/imrn/article-abstract/2024/8/6680/7331922}
}

@article{AlimVogrin2021GaussManinLieK3,
  author       = {Alim, Murad and Vogrin, Martin},
  title        = {Gauss–Manin Lie algebra of mirror elliptic K3 surfaces},
  journal      = {Mathematical Research Letters},
  volume       = {28},
  number       = {3},
  pages        = {637--663},
  year         = {2021},
  doi          = {10.4310/MRL.2021.v28.n3.a1},
  eprint       = {arXiv:1812.03185},
  archivePrefix= {arXiv},
  primaryClass = {math.AG},
  url          = {https://arxiv.org/abs/1812.03185}
}

\end{document}